\documentclass[12pt,reqno]{amsart}
\usepackage[headings]{fullpage}
\usepackage{amssymb,amsmath,amscd,bbm,tikz}
\usepackage{graphicx}
\usepackage{texdraw}
\usepackage{pb-diagram}
\usepackage[all,cmtip]{xy}
\usepackage{url}
\usepackage[bookmarks=true,%
    colorlinks=true,%
    linkcolor=blue,%
    citecolor=blue,%
    filecolor=blue,%
    menucolor=blue,%
    urlcolor=blue,%
    breaklinks=true]{hyperref}
\usepackage{slashed}    
\usepackage{verbatim}

\newtheorem{theorem}{Theorem}[section]
\theoremstyle{definition}
\newtheorem{proposition}[theorem]{Proposition}

\newtheorem{definition}[theorem]{Definition}

\newtheorem{conjecture}[theorem]{Conjecture}

\def\BN{\mathbbm N}

\def\BZ{\mathbbm Z}
\def\BQ{\mathbbm Q}

\def\BC{\mathbbm C}

\def\calM{\mathcal M}

\def\s{\sigma}
\def\la{\langle}
\def\ra{\rangle}

\def\SL{\mathrm{SL}}

\def\be{\begin{equation}}
\def\ee{\end{equation}}

\def\z{\zeta}
\def\DJ{\mathrm{DJ}}

\def\DK{\mathrm{DJ}}
\def\J{\mathrm{J}}
\def\H{\mathrm{H}}
\def\Zhat{\widehat{\BZ[q]}}
\def\bJ{\mathbf{J}}

\newcommand{\prin}[2]{\{#2\}_{#1}}
\newcommand{\hev}[2]{\operatorname{\theta}_{#1}\!\left(#2\right)}

\begin{document}
\title[The descendant colored Jones polynomials]{
       The descendant colored Jones polynomials}
\author{Stavros Garoufalidis}
\address{
  International Center for Mathematics, Department of Mathematics \\
  Southern University of Science and Technology \\
  Shenzhen, China \newline
  {\tt \url{http://people.mpim-bonn.mpg.de/stavros}}}
\email{stavros@mpim-bonn.mpg.de}
\author{Rinat Kashaev}
\address{Section de Math\'ematiques, Universit\'e de Gen\`eve \\
rue du Conseil-Général 7-9, 1205 Gen\`eve, Switzerland \newline
         {\tt \url{http://www.unige.ch/math/folks/kashaev}}}
\email{Rinat.Kashaev@unige.ch}

\date{19 June 2022}
\dedicatory{In memory of our friend, Vaughan Jones.}

\begin{abstract}
  We discuss two realizations of the colored Jones polynomials
  of a knot, one appearing in an unnoticed work of the second author in 1994
  on quantum R-matrices at roots of unity obtained from solutions of the
  pentagon identity, and another formulated in terms of a sequence of elements
  of the Habiro ring appearing in recent work of D. Zagier and the first
  author on the Refined Quantum Modularity Conjecture.
\end{abstract}

\maketitle

{\footnotesize
\tableofcontents
}


\section{Introduction}
\label{sec.intro}

The Jones polynomial~\cite{Jones} is a fascinating polynomial invariant
of knots with deep connections to the topology and geometry in dimension 3.
Its discovery in 1984, apart from revolutionizing Knot Theory, lead to a new area
of research, Quantum Topology, with applications and challenges that we will not
attempt to summarize here. The Jones polynomial and its colored  versions (a
sequence of polynomials, one for every irreducible finite dimensional representation
of the Lie algebra $\mathfrak{sl}_2$) are versatile invariants
with numerous interpretations, among them as partition functions of a 3-dimensional
Chern--Simons gauge theory~\cite{Witten:CS}. This sequence of Laurent
polynomials with integer coefficients (the so-called colored Jones polynomials of
a knot), introduced by Kirillov--Reshetikhin~\cite{KR} and Turaev~\cite{Tu:YB}, 
is not a random sequence of polynomials. Indeed, it was shown in 
~\cite{GL:qholonomic} that it is $q$-holonomic, i.e., that it satisfies a nontrivial
recursion relation with coefficients in $\BZ[q,q^n]$. Moreover, a canonically chosen
recursion is a knot invariant (the so-called $\hat A$-polynomial of a knot) which is
conjectured to determine the $\SL_2(\BC)$-character variety of the knot, viewed
from the boundary~\cite{Ga:AJ}. The colored Jones polynomial of a knot 
has appeared in equivalent forms in several occasions, such as the sequence of
Habiro polynomials of a knot~\cite{Ha,Habiro:WRT} and the $\mathfrak{sl}_2$
Akutsu--Deguchi--Ohtsuki invariants~\cite{ADO} which only recently were shown to
be equivalent to the colored Jones polynomials~\cite{Willetts}. 

An unexpected connection of the colored  Jones polynomials and hyperbolic
geometry~\cite{Thurston} arose from the Volume Conjecture~\cite{K97} for the Kashaev
invariant of a knot whose exponential growth rate is conjectured to be the suitably
normalized Volume of a hyperbolic knot. This conjecture, combined with the
identification of the Kashaev invariant with an evaluation of the colored  Jones
polynomials by Murakami--Murakami~\cite{MM}, linked the Jones polynomials with
hyperbolic geometry. 

In this paper we discuss two equivalent reformulations (or rather, realizations)
of the colored Jones polynomials of a knot: one is a sequence $\DJ^{(m)}(q)$
of elements of the Habiro ring defined in Equation~\eqref{Jm} below and shown to be
equivalent to the colored Jones polynomial of a knot in Proposition
~\ref{prop.equiv}. This sequence of elements of the Habiro ring defines the top row
of a matrix of knot invariants that arises in recent work of D.~Zagier and the first
author regarding the Refined Quantum Modularity Conjecture~\cite{GZ:kashaev}.  

A second reformulation of the colored Jones polynomials comes from 
unnoticed work of the second author in 1994 regarding quantum
R-matrices at roots of unity obtained from solutions of the pentagon
identity~\cite{Kashaev:algebraic}. These R-matrix invariants $\la K \ra_{N,n}$
of a knot $K$ are defined for a primitive $N$-th root of unity $\zeta$ (of strictly
positive order), and for a number $n \in \BZ/N\BZ$ in Theorem~\ref{thm.KNn} and
are conjectured to be equivalent to the colored Jones polynomials of a knot; see
Conjecture ~\ref{conj.2} below.

\subsection*{Acknowledgements} 

The authors wish to thank Jie Gu, Marcos Mari\~{n}o, Campbell Wheeler and Don Zagier
for many enlightening conversations, and the International Math Center in SUSTech
and the University of Geneva for their hospitality. The work of R.~K. is partially
supported by the Swiss National Science Foundation, the subsidies no~$200020\_192081$
and no~$200020\_200400$, and the Russian Science Foundation, the subsidy no~21-41-00018.


\section{The descendant colored Jones polynomials of a knot}
\label{sec.desc}

\subsection{A conjectured matrix-valued invariant from~\cite{GZ:kashaev}}

We begin by explaining the most recent reappearance of the colored Jones polynomials.
An extension of the Volume Conjecture is the Quantum Modularity
Conjecture of Zagier~\cite{Za:QMF} which, roughly speaking, discusses the
asymptotics of the Kashaev invariant at roots of unity. More recently, a
Refined Quantum Modularity Conjecture was proposed in~\cite{GZ:kashaev} which
suggests the existence of a matrix-valued invariant defined at (and near) roots
of unity.
These conjectured matrix-valued invariants (denoted by $\bJ(x)$ for $x \in \BQ/\BZ$)
were studied in~\cite[Sec.5]{GZ:kashaev} with explicit examples given for the
$4_1$ and the $5_2$ knots in Sections 7.1 and 7.2 of~\cite{GZ:kashaev}, respectively.

The conjectured matrix-valued knot invariant $\bJ$ (defined as a 1-periodic function
at the rationals, or alternatively as a function at the complex roots of unity),
has many interesting analytic, arithmetic properties and geometric properties which
are listed in detail in the introduction (Part 0) of~\cite{GZ:kashaev}.
The rows and columns of the matrix $\bJ$ are parametrized by the set of boundary
parabolic $\SL_2(\BC)$-representations of the knot (assuming the latter set is finite),
and this set always contains two distinguished elements, the trivial representation
$\s_0$, and the geometric representation $\s_1$ of a hyperbolic knot. The $\s_0$-column
of $\bJ$ is $(1,0,\dots,0)^t$ but the $\s_0$-row is very interesting. Its
$(\s_0,\s_1)$-entry is the Kashaev invariant of a knot, whereas the remaining entries
are expected to be elements of the Habiro ring (whose definition is recalled below),
tensored with the rational numbers.
Moreover, the entries of each row of the matrix are expected to extend to sequences
that satisfy a homogeneous recursion relation (if not the $\s_0$-row), or an
inhomogeneous recursion (if the $\s_0$-row). In other words, the matrix $\bJ$ is
expected to be a fundamental matrix solution of a $q$-holonomic module.

In~\cite[Sec.7.3,7.4]{Za:QMF} a method was proposed for defining all but the first row
of this matrix, given a suitable ideal triangulation of the knot. But the first row
of the above matrix remained elusive. The goal of this paper is to propose a definition
for this first row and verify some of its conjectured properties. After a linear
change of variables, the first row of the matrix $\bJ$ is conjecturally equal to one
of the best-known quantum knot invariants, namely the colored Jones polynomials; see
Conjecture~\ref{conj.1} below.

\subsection{Definition of the descendants}

In this section we extend the 2-parameter colored Jones function $\J_n(q)$ of a knot
in 3-space into a 3-parameter \emph{descendant colored Jones} function $\DJ_n^{(m)}(q)$
for $n \geq 0$, $m \in \BZ$, which
\begin{itemize}
\item 
  specializes  to the colored  Jones polynomials $\J_n(q)$ when $m=0$ and $n \geq 1$;
\item 
  specializes to a sequence of elements of the Habiro ring $\DK^{(m)}(q)$ when $n=0$;
\item 
  is determined by either of the above  specializations;
\item 
  is determined by a 3-parameter function $\DJ^{(m)}(x,q)$ by
  $\DJ_n^{(m)}(q)=\DJ^{(m)}(q^n,q)$ for all $n \geq 1$.
\end{itemize}
In other words, we have a commutative diagram
{
\be
\label{3term}
\begin{diagram}
  \node[2]{\DJ^{(m)}_n(q)} \arrow{sw,l}{m=0} \arrow{se,l}{n=0} \\
  \node{\J_n(q)} \node[2]{\DK^{(m)}(q)} \\
\node[2]{\DJ^{(m)}(x,q)} \arrow[2]{n,r}{x=q^n} \arrow{nw} \arrow{ne}   
\end{diagram}
\ee
}
To define the functions appearing in this diagram, we start with the 
colored  Jones polynomials $\J^K_n(q) \in \BZ[q^\pm]$ of a knot $K$ (mostly
omitted from the notation), colored  by the $n$-dimensional irreducible representation
of $\mathfrak{sl}_2$ for $n \geq 1$ and normalized to be 1 at the unknot~\cite{Tu:book}.
In~\cite{Habiro:WRT} Habiro proved that the colored  Jones polynomials can be written
in the form
\be
\label{JH}
\J_n(q) = \sum_{k=0}^\infty c_{n,k}(q) \H_k(q), \qquad 
c_{n,k}(q) = q^{-kn} (q^{n+1};q)_k (q^{n-1};q^{-1})_k 
\ee
where $\H_k(q) \in \BZ[q^\pm]$ for all $k \geq 0$ and
$c_{n,k}(q)$ is the cyclotomic kernel which is independent of the knot and vanishes
when $k \geq n \geq 1$. Equation~\eqref{JH} can be inverted~\cite[Lem.~6.1]{Habiro:WRT}
\be
\label{HJ}
\H_k(q) = \sum_{n=1}^{k+1} \gamma_{k,n}(q) \J_n(q), \qquad \gamma_{k,n}(q) = 
(-1)^{k-n-1}
\frac{q^{\frac{k(k+3)+n(n-3)}{2}+1}(1-q^{n})(1-q^{2n})}{(q;q)_{k+n+1} (q;q)_{k-n+1}}
\ee
where $\gamma_{k,n}(q)=0$ for $n \geq k+2$ and $k \geq 0$. 
There are three important features in Habiro's expansion~\eqref{JH} of the colored
Jones polynomials.
\begin{itemize}
  \item[(i)] The Habiro polynomials $\H_k(q)$ which a priori lie in $\BQ(q)$
(as follows from~\eqref{HJ}), actually lie in $\BZ[q^\pm]$.
\item[(ii)] Setting $n=0$ in the right hand side of~\eqref{JH} gives a well-defined
  element of the Habiro ring $\Zhat = \varprojlim_n \BZ[q]/((q;q)_n)$.
\item[(ii)] The dependence on the color comes only through the cyclotomic kernel
  $c_{n,k}(q)$. Thus, one can replace $q^n$ by a variable $x$ and define an element
  $\J(x,q)$ of the colored  Habiro ring
  $\Lambda=\varprojlim_n \BZ[q^\pm][x+x^{-1}-2]/(c_n(x,q))$,
  so that $\J(q^n,q)=\J_n(q)$ for $n \geq 1$, where
  \be
  \label{JHx}
\J(x,q) = \sum_{k=0}^\infty c_{k}(x,q) \H_k(q), \qquad
c_{k}(x,q) = x^{-k}(q x;q)_k (q^{-1}x;q^{-1})_k \,.
  \ee
\end{itemize}


We now have all the ingredients to define the descendant colored Jones function.
\begin{definition}
  \label{def.desc}
For an integer $m$ and for $n \geq 0$, we define
\be
\label{Jmn}
\DJ^{(m)}_n(q) =  \sum_{k=0}^\infty c_{n,k}(q) \H_k(q) \, q^{km},
\qquad
\DJ^{(m)}(x,q) = \sum_{k=0}^\infty c_{k}(x,q) \H_k(q) q^{km} 
\ee
and
\be
\label{Jm}
\DK^{(m)}(q) =  \sum_{k=0}^\infty (q;q)_k (q^{-1};q^{-1})_k  \H_k(q) \, q^{km} \,.
\ee
\end{definition}

\subsection{Properties of the descendants}

From their very definition, the $0$-th descendants $\DJ^{(0)}_n(q)$ and $\DJ^{(0)}(q)$
are nothing but the colored Jones polynomial $\J_n(q)$ and the Kashaev invariant of
the knot. Thus, one may call $\DJ^{(m)}_n(q)$ and $\DJ^{(m)}(q)$ the descendant
colored Jones and the descendant Kashaev invariant of a knot, respectively.

Note further that for all integers $m$, the 3-variable invariant $\DJ^{(m)}(x,q)$ is
an element of the colored  Habiro ring, and, for all $n \geq 0$, it
satisfies $\DJ^{(m)}(q^n,q)=\DJ^{(m)}_n(q)$.

Note moreover that $\DJ^{(m)}_n(q) \in \BZ[q^\pm]$ for $n \geq 1$ whereas
$\DK^{(m)}(q) \in \Zhat$. In a sense, $\DK^{(m)}(q)$ is a twisted Fourier transform of
$\J_n(q)$.

The next proposition describes the equivalence among these knot invariants. 
Let $\mu$ denote the set of complex roots of unity.

\begin{proposition}
  \label{prop.equiv}
  Each family of invariants $\{\J_n(q)| n \in \BZ_{\geq 0}\}$,
  $\{\DJ^{(m)}_n(q)| n \in \BZ_{\geq 0}, m \in \BZ\}$, $\{\DJ^{(m)}(x,q)| m \in \BZ\}$,
  $\{\DK^{(m)}(q)|m \in \BZ\}$, $\{\J_n(\z)|n \in \BZ_{\geq 0}, \z \in \mu\}$,
  $\{\DJ^{(m)}_n(\z)|n \in \BZ_{\geq 0}, \z \in \mu\}$, $\{\DJ^{(m)}(x,\z)
  m \in \BZ, \z \in \mu\}$, $\{\DK^{(m)}(\z)| m \in \BZ, \z \in \mu\}$ determines all
  other families. 
\end{proposition}

\begin{proof} 
We first discuss the equivalence of invariants in the left triangle of
Diagram~\eqref{3term}. Clearly, the descendant colored  Jones polynomials
$\DJ^{(m)}_n(q)$ specialize to the colored  Jones polynomials $\J_n(q)$ when $m=0$.
Conversely, $\J_n(q)$ determines $\H_k(q)$ by Equation~\eqref{JH} and hence determines
$\DJ^{(m)}_n(q)$ by Equation~\eqref{Jmn}. 

The evaluation maps $\Lambda \to \Zhat$ given by $x \mapsto q^n$ for all $n$ give
an injection of $\Lambda$ to $\Zhat^\BN$ as was shown in~\cite{Habiro:WRT}. It
follows that $\DJ^{(m)}_n(q)$ determines $\DJ^{(m)}(x,q)$.  

This shows that each one of the invariants $\J_n(q)$, $\DJ^{(m)}_n(q)$,
$\DJ^{(m)}(x,q)$ determines all others. Moreover, since the collection of evaluation maps
$\Zhat \to \overline{\BQ}$ given by $q \mapsto \z$ uniquely determines an element
of the Habiro ring~\cite{Habiro:completion}, and the same holds for the colored  Habiro
ring~\cite{Habiro:WRT}, it follows that each one of the invariants $\J_n(q)$,
$\DJ^{(m)}_n(q)$, $\DJ^{(m)}(x,q)$, $\J_n(\z)$, $\DJ^{(m)}_n(\z)$, $\DJ^{(m)}(x,\z)$
determines all others. 

We now discuss the equivalence of invariants in the right triangle of
Diagram~\eqref{3term}. The descendant colored  Jones polynomials $\DJ^{(m)}_n(q)$
specialize to $\DK^{(m)}(q)$ when $n=0$, which specialize to 
$\DJ^{(m)}(\z)$ for any complex root of unity $\z$. Moreover, by
Equation~\eqref{Jmn} and the inversion of the discrete Fourier transform, it follows
that $\DJ^{(m)}(\z)$ determines $(\z;\z)_k (\z^{-1};\z^{-1})_k \, \H_k(\z)$, and hence,
$(q;q)_k (q^{-1};q^{-1})_k \, \H_k(q)$ (since the latter are in $\BZ[q^\pm]$),
and hence $\H_k(q)$, and hence $\DJ^{(m)}_n(q)$. 
\end{proof}

The next proposition concerns the $q$-holonomic properties of the descendant
colored Jones functions. For a detailed definition of $q$-holonomic functions in
several variables and their closure properties, we refer the reader to
~\cite{Zeil:holo, AB} and also to~\cite{Sab,GL:survey}. Roughly speaking,
$q$-holonomic functions have annihilating ideals (i.e., ideals of linear
$q$-difference equations) of maximal dimension.

\begin{proposition}
  \label{prop.qholo}
  $\DJ^{(m)}_n(q)$, $\DJ^{(m)}(x,q)$ and $\DK^{(m)}(q)$ are $q$-holonomic functions of
  $(m,n)$, $(m,x)$ and $m$, respectively. Linear $q$-difference equations for
  these functions can be obtained from one for the Habiro polynomials $\H_k(q)$. 
\end{proposition}

\begin{proof}
  The first part follows from the closure properties of $q$-holonomic functions
  explained in detail in~\cite{GL:survey}, and the fact that
  the colored Jones polynomial (and hence,
  the sequence of Habiro polynomials) is $q$-holonomic~\cite{GL:qholonomic} and 
  $c_{n,k}(q)$ and $g_{k,n}(q)$ are $q$-holonomic (in fact $q$-proper hypergeometric)
  functions. Using a computer implementation given by
  Koutschan~\cite{Koutschan, Koutschan:HF}, we can obtain a linear $q$-difference
  equation for $\DK^{(m)}(q)$ given one for $\H_k(q)$. 
  The obtained linear $q$-difference equations may not be of the smallest order.
\end{proof}

The above proposition implies that the annihilating ideal of the 3-variable invariant
$\DJ^{(m)}(x,q)$ is a knot invariant, and so is a minimal recursion of 
$\DJ^{(m)}(x,q)$ with respect to $x$ or $m$. Moreover, a minimal recursion with
respect to $m$ can be computed from a recursion for the Habiro polynomials $\H_k$
as well as some of their initial values. This is illustrated in
Section~\ref{sec.examples} below.

The annihilating polynomial of $\DJ^{(m)}(x,q)$ with respect to $m$ has an 
excess degree whose relation to the character variety of $\SL_2(\BC)$-representations
is a challenging question discussed in Section~\ref{sec.examples} below.


To effectively apply Proposition~\ref{prop.qholo} we need to know a recursion for
the Habiro
polynomials of a knot. Such a recursion can be obtained from a $q$-hypergeometric
formula for the Habiro polynomials, as was done in~\cite{GS:Cpoly} for all twist
knots using a formula for the Habiro polynomials given in~\cite{Ha,Masbaum}.
Alternatively, one can use a recursion for $\J_n(q)$, Equation~\eqref{HJ}
and the methods of~\cite{Koutschan:HF} to obtain a recursion for the Habiro polynomials.
Although the above mentioned computer implementation in theory succeeds (and does
so with a certified computation), in practice it may not terminate. As a concrete
example, the colored  Jones polynomials of the $(-2,3,7)$ pretzel knot satisfy a
guessed inhomogeneous sixth order $q$-difference equation~\cite{GK:pretzel}, but
a direct attempt to find a recursion for its Habiro polynomials did not terminate.
Instead, one can use the recursion for the colored  Jones to compute several thousand
values of them (modulo a prime) and Equation~\eqref{HJ} to do the same for the Habiro
polynomials, and with some hints for the shape of the Newton polygon to guess a
recursion for the Habiro polynomials. This is explained in detail for the
$(-2,3,7)$ pretzel knot and for several other knots in forthcoming
work~\cite{GK:237habiro}.


Let us comment on the chirality properties of the descendant colored Jones
functions. It is well-known that the Jones polynomial (and its colored  versions)
are chiral invariants of knots. Explicitly, if $K^*$ denotes the mirror of a knot
$K$, then $\J_n^{K^*}(q)=\J_n(q^{-1})$ which implies that $\H_k^{K^*}(q)=\H_k(q^{-1})$
and $\DJ^{K^*,(m)}(x,q)=\DJ^{K,(-m)}(x^{-1},q^{-1})$. 

\subsection{A conjecture}

We end this section with a conjecture. Let us denote by $r+1$ the size of the
matrix $\bJ$. 

\begin{conjecture}
  \label{conj.1}
  The top row of the matrix $\bJ$ of the conjectured knot invariants~\cite{GZ:kashaev}
  is a rational linear combination of $r+1$ consecutive
  terms of the descendant invariant $\DK^{(m)}(q)$.
\end{conjecture}  

\subsection{Examples}
\label{sec.examples}

In this section we give explicit formulas for the descendant colored  Jones
polynomials and their recursions for the trefoil (a non-hyperbolic knot) and 
for the two simplest hyperbolic knots, the $4_1$ and $5_2$ knots.

For the $3_1$ knot, we have $\H_k^{3_1}(q)=(-1)^k q^{\frac{1}{2}k(k+3)}$
for all $k \geq 0$~\cite{Ha,Masbaum}. It follows from Equation~\eqref{Jmn}
that 
the descendant colored  Jones polynomials $\DJ^{3_1,(m)}(x,q)$
satisfy the inhomogeneous linear $q$-difference equation
\be
\label{rec31}
-q^{m+3} \,\DJ^{3_1,(m+2)}(x,q) + (x+x^{-1}) q^{m+2}\, \DJ^{3_1,(m+1)}(x,q)
+ (1-q^{m+1}) \, \DJ^{3_1,(m)}(x,q) =1 
\ee
for $m \in \BZ$. 
We can write the above $q$-difference equation in operator form by introducing
the operators $S_m$ and $Q_m$ which act on $f(m)$ by
$(S_m f)(m)=f(m+1)$, $(Q_m f)(m)=q^m f(m)$ and satisfy the $q$-commutation relation
$S_m Q_m = q Q_m S_m$. Then, Equation~\eqref{rec31} takes the form
$B^{3_1} \DJ^{3_1,(m)} = 1$ where
\be
B^{3_1}(S_m,Q_m,x,q)=-q^3 Q_m S_m^2 +
(x+x^{-1}) q^2 Q_m S_m + (1-q Q_m) \,.
\ee

We next consider the $4_1$ knot, whose Habiro polynomials are given by
$\H_k^{4_1}(q)=1$ for all $k \geq 0$, hence
\be
\DJ^{4_1,(m)}(x,q) = \sum_{k=0}^\infty x^{-k} (qx;q)_k (q^{-1}x;q^{-1})_k \, q^{km} \,.
\ee
When $x=1$, these coincide (after a suitable $\BQ[q^\pm]$ linear combination)
with the elements in the top row of an experimentally found matrix in~\cite{GZ:kashaev}.
The descendants satisfy the inhomogeneous linear
$q$-difference equation
\be
\label{rec41}
q^{m+1}\,\DJ^{4_1,(m+1)}(x,q) + (1-(x+x^{-1}) q^m)\, \DJ^{4_1,(m)}(x,q)+ q^{m-1}
\DJ^{4_1,(m-1)}(x,q) =1 
\ee
for $m \in \BZ$. Equation~\eqref{rec41} takes the form $B^{4_1} \DJ^{4_1,(m)} = 1$ where
\be
B^{4_1}(S_m,Q_m,x,q)=q Q_m S_m^2 + (1-(x+x^{-1}) Q_m) S_m + Q_m^{-1} S_m^{-1} \,.
\ee
Note that after setting $q=1$ and $x=1$ and $Q_m=1$ and renaming $S_m$ to be $L$, we
have $B^{4_1}(L,1,1,1)=L^2 - L + 1$ has roots generating the two embeddings
of the trace field $\BQ(\sqrt{-3})$ of the $4_1$ knot. 

A more interesting example is the $5_2$ knot, where the descendant colored  Jones
polynomials (or their evaluation at $q=1$) is a one-parameter $q$-holonomic module,
whereas the two-dimensional state-sum formula for the Kashaev invariant leads
to a two-parameter $q$-holonomic module considered in~\cite{GZ:kashaev} in relation to
the Refined Quantum Modularity Conjecture. The coincidence of the spanning space of
the two modules is no longer obvious, but appears to be true. Another feature of this
knot is that the recursion for the descendant colored  Jones polynomials has higher
degree (precisely, two more) than the degree of the recursion of the colored  Jones
polynomials, and the degree of the $A$-polynomial of the knot. This excess degree and
its relation to the character variety of $\SL_2(\BC)$-representations is
a challenging question discussed further below. 

We first recall the Habiro polynomials of the $5_2$. The latter is the twist knot
$K_2$ and its Habiro polynomials were given by Habiro~\cite{Ha}; see also
Masbaum~\cite{Masbaum} for a detailed discussion.
In fact, the Habiro polynomials of $5_2$ are given explicitly by
\be
\label{Hk52}
\H_k^{5_2}(q) = (-1)^k q^{\frac{1}{2}k(k+3)} \sum_{s=0}^k q^{s(s+1)} \binom{k}{s}_q
\ee
where $\binom{a}{b}_q = (q;q)_a/((q;q)_b (q;q)_{b-a})$ is the $q$-binomial function.
In~\cite{GS:Cpoly}, it was shown that $\H_k=\H_k^{5_2}(q)$ satisfies the linear
$q$-difference equation
\be
\label{recHk52}
\H_{k+2}^{5_2}(q) + q^{3 + k} (1 + q - q^{2 + k} + q^{4 + 2 k}) \H_{k+1}^{5_2}(q)
-q^{6 + 2 k} (-1 + q^{1 + k}) \H_k^{5_2}(q) =0, \qquad (k \geq 0)
\ee
with initial conditions $\H_k^{5_2}(q)=0$ for $k<0$ and $\H_0^{5_2}(q)=1$.
Actually, the above recursion is valid for all integers if we replace the right
hand side of it by $\delta_{k+2,0}$. This, combined with Equation~\eqref{Jmn}
and~\cite{Koutschan:HF}, gives that $\DJ^{(m)}=\DJ^{5_2,(m)}(x,q)$ satisfies the
linear $q$-difference equation
{\tiny
\begin{multline}
\label{recJm52x}
  (-1 + q^{1 + m}) (-1 + q^{2 + m}) x^2 \DJ^{(m)} - 
  q^{2 + m} (-1 + q^{2 + m}) x (1 + q + x + (1+ q) x^2) \DJ^{(1 + m)} \\
  + q^{3 + m} (q^{3 + m} +(- 1 + q^{2 + m} + q^{3 + m}) x +(- 2  - 
    q + q^{2 + m} + 2 q^{3 + m} + q^{4 + m}) x^2 +(-1 + 
    q^{2 + m} + q^{3 + m}) x^3 + q^{3 + m} x^4) \DJ^{(2 + m)} \\
    - 
 q^{4 + m} (q^{3 + m} +(-1 + q^{3 + m} + q^{4 + m}) x +(-1 + 
    q^{2 + m}  + 2 q^{3 + m}  + q^{4 + m} ) x^2 +(-1 + 
    q^{3 + m} + q^{4 + m}) x^3 + q^{3 + m} x^4) \DJ^{(3 + m)} \\
    + q^{5 + m} x (q^{3 + m} + q^{4 + m} +(-1 + q^{4 + m}) x + 
    (q^{3 + m} + q^{4 + m}) x^2) \DJ^{(4 + m)} - 
    q^{10 + 2 m} x^2 \DJ^{(5 + m)}
\\ = x (q^{2 + m} + q^{4 + m} + (1 - q^{1 + m} - 2 q^{3 + m} - 
    q^{5 + m}) x + (q^{2 + m} + q^{4 + m}) x^2) \H_0(q) + 
 q^m x (1 - x q^{-1}) (1 - q x) \H_1(q) \,.
\end{multline}
}

\noindent
Using the values $\H_0(q)=1$, $\H_1(q)=-q^2-q^4$, it follows that the
right hand side of the above recursion is $x^2$ for all $m$. 
Writing the above equation in operator form $B^{5_2} \DJ^{5_2,(m)} =x^2$,
where $B^{5_2}=B^{5_2}(S_m,Q_m,x,q)$, we obtain that
$B^{5_2}(L,1,1,1)=-L^2 (-5 + 7 L - 4 L^2 + L^3)$, and three
nonzero roots of the above polynomial generate the three embeddings of the trace
field of the $5_2$ knot. The latter is a cubic field of discriminant $-23$.

Although Equation~\eqref{recJm52x} is fifth order and inhomogeneous, we claim that
the $\BZ[q^\pm]$-module of the colored  Habiro ring generated by $\DJ^{(m)}(x,q)$
for all $m \geq 0$ equals to the module $\calM$ spanned by
$\{1,\DJ^{(0)}(x,q),\DJ^{(1)}(x,q),\DJ^{(2)}(x,q)\}$.
Indeed, since the coefficient of $\DJ^{(m)}$ (resp. $\DJ^{(m+1)}$) in~\eqref{recJm52x}
vanishes when $m=-2$ and $m=-1$, (resp. $m=-2$) the specialization of~\eqref{recJm52x}
to $m=-2$ and to $m=-1$ gives that $\DJ^{(3)}(x,q)$ (resp. $\DJ^{(4)}(x,q)$) is in
$\calM$. 
Given this, Equation~\eqref{recJm52x} then implies that $\DJ^{(m)}(x,q) \in \calM$
for all natural numbers $m$. On the other hand, $\DJ^{(m)}(x,q)$ for $m=-1$ and
$m=-2$ do not seem to lie in $\calM$. 

We now recall a second $q$-holonomic module which is motivated by
the (projection-dependent) formula for the Kashaev invariant of $5_2$
given in~\cite[Eqn.~(2.3)]{K97}

\be
\label{K52}
\J(q) = \sum_{k,l \geq 0} q^{-(k+l+1)l} \frac{(q;q)_{k+l}^2}{(q^{-1};q^{-1})_l}
= \sum_{k,l \geq 0} (-1)^l q^{-\frac{1}{2}(2k+l+1)l} \frac{(q;q)_{k+l}^2}{(q;q)_l} \,.
\ee
(Warning: the above formula is $q$ times the evaluation of the $n$-th colored  Jones
polynomial at the $n$-th root of unity.) Given this formula, we can define a 2-parameter
family of descendants (elements of the Habiro ring) by
\be
\label{K52d}
\DJ_{a,b}(q) = \sum_{k,l \geq 0} (-1)^l q^{-\frac{1}{2}(2k+l+1)l + a k + b l}
\frac{(q;q)_{k+l}^2}{(q;q)_l} \qquad (a,b \in \BZ) \,.
\ee
Then, $\DJ_{a,b}(q)$ is a $q$-holonomic module in two variables.
Experimentally, it appears that the span over $\BQ(q)$ of $\DJ_{a,b}(q)$ for
$(a,b)$ in some cone coincides with the span of $\{1,\DJ^{(0)},\DJ^{(1)},\DJ^{(2)}\}$.
For instance, we have:
\begin{align*}
\DJ_{1,0} &= (3 - q^{-1}) \DJ^{( 0)} + (1 - 3 q) \DJ^{( 1)} + q^2 \DJ^{( 2)}
\\
\DJ_{-1,0} &= 3 q \DJ^{( 0)} - 3 q^2 \DJ^{( 1)} + q^3 \DJ^{( 2)}
\\
\DJ_{1,-1} &= 
(3 + q^{-2} - q^{-1}) \DJ^{( 0)} + (1 - q^{-1} - 3 q) \DJ^{( 1)} + 
 q^2 \DJ^{( 2)} -q^{-2} + q^{-1} 
\\
\DJ_{0,-1} &=
-\DJ^{( 0)} + q \DJ^{( 1)} + 1
\\
\DJ_{-1,-1} &=
2 q \DJ^{( 0)} - q^2 \DJ^{( 1)} \,.
\end{align*}


In the examples given above for the $3_1$, $4_1$ and $5_2$ knots, we observed two
phenomena for the annihilating polynomial $B(S_m,Q_m,x,q)$ with respect to $m$ of
the descendant colored  Jones polynomials, namely:
\begin{itemize}
\item[(a)]
  The degree of $B(S_m,Q_m,x,q)$ is greater than or equal to the degree of
  the $A$-polynomial of a knot~\cite{CCGLS} (which conjecturally coincides with
  the degree of the minimal inhomogeneous recursion of the colored  Jones polynomials).
  Let $\delta$ denote this difference.
\item[(b)]
  $B(S_m,Q_m,x,q)$ is monic and the coefficient of $S_m^j$ in $B(S_m,Q_m,x,q)$
  vanishes when $m=-\delta,-\delta+1,\dots, -j-1$ for $j=0,\dots,\delta-1$. This
  implies that the $\BZ[q^\pm]$-span of $\DJ^{(m)}(x,q)$ for $m \geq 0$ equals
  to the span of $\DJ^{(m)}(x,q)$ for $m=0,\dots,\text{deg}(A)-1$.
\end{itemize}


The above observations hold for all twist knots $K_p$ considered
in~\cite[Fig.~3]{Masbaum}, where $K_1=3_1$, $K_{-1}=4_1$ and $K_2=5_2$. Their
$A$-polynomials were computed by Hoste--Shanahan~\cite{Hoste-Shanahan:A} and the
$\widehat A$ polynomials (i.e. a minimal recursion for the colored  Jones polynomials)
was computed in~\cite{Le:Atwist}. The
 recursions for the Habiro polynomials of the twist knots was given
in~\cite{GS:Cpoly}, and, using this, one can obtain recursions for the descendant
colored  Jones polynomials with respect to $m$. An explicit computation shows that
the $B$-polynomials of twist knots satisfy the above mentioned two properties.
Moreover, the degrees of the $A$-polynomials, the
$B$-polynomials and the excess numbers of twist knots are given by
\be
\label{degKp}
\begin{aligned}
\deg(A_{K_p}) &=  2p-1, & \deg(B_{K_p}) &=3p-1, &
\delta(K_p) =&p, & (p>0) \\
\deg(A_{K_p}) &=  2|p|, & \deg(B_{K_p}) &=3|p|-1, &
\delta(K_p) =&|p|-1 & (p<0) \,.
\end{aligned}
\ee
We mention parenthetically that the degree of the $A$-polynomials of the twist
knots coincide with the degrees of their trace fields~\cite{Hoste-Shanahan:trace}.
The excess degree of a knot and its relation to quantum invariants, their
asymptotics, and their connection with the work of Gukov et al~\cite{Gukov-Manolescu}
is studied in detail by Campbell Wheeler in his upcoming thesis~\cite{Wheeler:thesis}
and in forthcoming work~\cite{GGMW}.


\section{The colored  Jones polynomials at roots of unity}
\label{sec.cjN}

This section fills a historical gap where we discuss a previously unnoticed conjectural
relationship between the colored Jones polynomials and some quantum R-matrices at
roots of unity considered by the second author in~\cite{Kashaev:algebraic}. Namely,
we show that the R-matrices over the vector space $\BC^N$ depending on a primitive
$N$-th root of unity $\z$ and indexed by an integer parameter $0< m<N$ actually give
knot invariants which conjecturally coincide with the $(m+1)$-colored Jones polynomial
evaluated at  $q=\z$. In particular, we show that a specific choice of the
discrete parameter $m=N-1$ gives the Kashaev invariant $\langle K \rangle_N$ of a knot,
which, in its turn, appears to be equal to the $N$-colored Jones polynomial
$J_N(\z)$~\cite{MM}. The other choices of $m$ lead to knot invariants
$\langle K \rangle_{N,m}$ for $m\in \{0,\dots, N-1\}$ (or even links, with each
component colored by its own integer parameter).
 
In~\cite{K94}, the second author assigned a state-sum to a root of unity and to
a pair consisting of a triangulated compact closed oriented 3-manifold and a link
represented by a Hamiltonian sub-complex, and showed that the state-sum is invariant
under some natural Pachner moves. The state-sum depended on a tetrahedral weight
function, from which one can obtain an R-matrix that satisfies the Yang--Baxter
equation. This led to an invariant of knots in the 3-sphere that depends to a root
of unity, where the knot in question is a Hamiltonian path of a triangulation of
the 3-sphere. This is how the R-matrix
of~\cite{K95} was found and explained (without pictures), see~\cite[Sec.~4]{K95}.
The R-matrices of~\cite{Kashaev:algebraic}, which depend on discrete and continuous
spectral parameters, are expected to generalize  the R-matrix of~\cite{K95}.
The continuous spectral parameter apparently plays no role in constructing knot
invariants, while the discrete spectral parameter seems to take care of the color
in the colored Jones polynomials. 

We should remark that for all values of the colors, this R-matrix
of~\cite{Kashaev:algebraic} is over the same vector space $\BC^N$ and appears
nontrivial even when the color is trivial. This is in contrast to the standard
(quantum group) definition of the colored Jones polynomials, where the size of the
R-matrix depends on the color and is trivial for the trivial color.

\subsection{R-matrices with a spectral parameter}
\label{sub.R}

We begin our discussion of R-matrices by introducing some useful notation. 
For a positive integer $n\in\BZ_{>0}$, we define 
\begin{equation}
  \label{double-angle-brackets}
  (x)_n:=\prod_{k=1}^n(1-x^k), \qquad
  \langle\!\langle x\rangle\!\rangle_n:=\frac{1-x^n}{(1-x)n} 
\end{equation}
with the convention $(x)_0:=1$.

For a fixed root of unity $\z$ of order $N=\operatorname{ord}(\z)\in\BZ_{>0}$,
we define rational functions $w(x|n)\in \BC(x)$ for all integers
$n\in\BZ$ by the recurrence relation
\begin{equation}
  \label{recurrence-formula-w}
w(x|n)(1-x\z^n)=w(x|n-1),\quad w(x| 0)=1.
\end{equation}
One has the addition property
\begin{equation}\label{addition-formula-w}
w(x|m+n)=w(x|m)w(x\z^m|n)\qquad (m,n\in\BZ)
\end{equation}
which, by taking into account the equality $w(x|N)=1/(1-x^N)$,
gives the quasi-periodicity property
\begin{equation}
(1-x^N)w(x|n+N)=w(x|n)\qquad (n\in\BZ).
\end{equation}
Notice the relation to the $q$-Pochhammer symbols with $q=\z$
\begin{equation}
w(x|n)=\frac1{(x\z;\z)_n}\qquad (n\in\BZ_{\ge0}).
\end{equation}

It follows that $w(x|n)$ is a rational function of $x$ with coefficients in the
cyclotomic field $F_N:=\BQ(\z)$ whose poles and zeros (if any) are integer
powers of $\z$. We will call such rational functions, and more generally matrices
of such rational functions, $N$-cyclotomic. 

In~\cite{Kashaev:algebraic}, the second author considered the $N$-cyclotomic
matrix $r(x;m,n) \in \operatorname{Mat}_{N^2}(\BC(x))$
defined by the formula
\begin{equation}
  \label{r-matrix-coeff-diff-form}
  \langle i,j|r(x;m,n)|k,l\rangle=
  \langle\!\langle x\rangle\!\rangle_N\z^{(i-k+n)(l-j)}
  \frac{w(x/\z|j-i-m)w(x|l-k+n)}{w(x/\z |j-k+n-m)w(x/\z |l-i)}
\end{equation}
(where the left hand-side is the bra-ket notation for the entries of a matrix) 
and proved that it satisfies the Yang--Baxter equation
\begin{multline}
  \label{YB}
r_{1,2}(x;m_1,m_2)r_{1,3}(xy;m_1,m_3)r_{2,3}(y;m_2,m_3)\\
=r_{2,3}(y;m_2,m_3)r_{1,3}(xy;m_1,m_3)r_{1,2}(x;m_1,m_2)
\end{multline}
where, along with continuous spectral parameters $x$ and $y$, the integers
$m_1,m_2,m_3$ play the role of discrete spectral parameters. Here, $r_{1,2}$,
$r_{1,3}$ and $r_{2,3}$ denote the $N$-cyclotomic matrices in
$\operatorname{Mat}_{N^3}(\BC(x))$ obtained by interpreting $r$ as an endomorphism
of an $N$-dimensional space $V^{\otimes 2}$, equipped with a basis, and $r_{i,j}$ as the
corresponding endomorphisms of $V^{\otimes 3}$, acting nontrivially on the $i,j$ copies
of $V$ in $V^{\otimes 3}$ and trivially on the remaining copy.

Aside from being a solution to the Yang--Baxter equation with a spectral parameter,
a somehow surprising fact is that $r(x;m,n)$ is regular (i.e. it has a finite limit)
at $x=1$. This follows from the existence of the gauge conjugating matrix
$h(y,m) \in \operatorname{Mat}_{N}(\BC(y))$ defined by 
\begin{equation}
\langle i|h(y,m)|j\rangle=\z^{(j-i)m}\langle\!\langle y \,\z^{j-i}\rangle\!\rangle_N \,.
\end{equation}
The entries of the matrix $h(y,m)$ are, in fact, polynomials in $y$, it
satisfies the multiplicative property
\begin{equation}
\label{h1}
  h(x,m)h(y,m)=h(xy,m) 
\end{equation}
and it enters the important gauge symmetry transformation formulae for the
r-matrix~\eqref{r-matrix-coeff-diff-form}:
\begin{equation}
\label{h2}  
h_1(y,m+1)r(x;m,n)h_1(y^{-1},m+1)=r(xy;m,n)=
h_2(y^{-1},0)r(x;m,n)h_2(y,0) \,.
\end{equation}
Here, $h_1$ and $h_2$ denote matrices in $\operatorname{Mat}_{N^2}(\BC(y))$
obtained by interpreting $h$ as an endomorphism of an $N$ dimensional space $V$
with a basis, and $h_j$ as corresponding endomorphisms of $V^{\otimes 2}$
obtained by acting on the $j$-th copy of $V$ in $V^{\otimes 2}$ by $h$ and on
the remaining copy of $V$ by identity. 

The above gauge symmetry transformation formulae, together with the fact that
$h(y,m)$ has polynomial entries in $y$, imply easily that $r(x;m,n)$
is regular at $x=1$. An additional calculation shows that 
\begin{equation}
\label{r1}
  \langle i,j|r(1;m,n)|k,l\rangle=V_{i,j-m,k-n,l}(\z)\z^{k-l-n+(k-i-n)m}
\end{equation}
where
\begin{equation}
  V_{i,j,k,l}(\z):=\frac{N\hev{N}{\prin{N}{j-i-1}+\prin{N}{l-k}}
    \hev{N}{\prin{N}{i-l}+\prin{N}{k-j}}}{
    (\bar\z)_{\prin{N}{j-i-1}}(\z)_{\prin{N}{i-l}}
    (\bar\z)_{\prin{N}{l-k}}(\z)_{\prin{N}{k-j}}} \,
\end{equation}
and where $\bar\z$ denotes the complex conjugate of the complex number $\z$
and for an integer $k$, $\prin{N}{k}$ denotes the unique integer such that
\begin{equation}
\prin{N}{k}\equiv k\pmod N, \quad 0\le \prin{N}{k}<N \,,
\end{equation}
and
\begin{equation}
  \hev{N}{k}:=\delta_{k,\prin{N}{k}}
\end{equation}
where $\delta$ is the standard Kronecker delta symbol.
The regularity of the matrix $r(x;m,n)$ is reminiscent to the $p$-adic valuation
of factorials in Landau's work, see for instance~\cite{RV,Sounda} and references
therein. 

A consequence of formula~\eqref{r1} is that the rows and columns of $r(1;m,n)$
are naturally indexed by elements of the additive group $\BZ/N\BZ$, as well as
the discrete spectral parameters are in $\BZ/N\BZ$. 

\subsection{Proof of the gauge symmetry equations}
\label{sub.proofh1h2}

In this section we give the proofs of equations~\eqref{h1}
and~\eqref{h2} which were stated in~\cite{Kashaev:algebraic}, but proofs
were ommitted.

We begin with Equation~\eqref{h1}. Writing out its left hand side in matrix
coefficients (indexed by $i,j\in\{0,\ldots,N-1\}$), we have
\begin{equation}
\langle i|h(x,m)h(y,m)|j\rangle=\sum_{k=0}^{N-1}
\langle i|h(x,m)|k\rangle\langle k|h(y,m)|j\rangle
=\z^{(j-i)m} \sum_{k=0}^{N-1}
\langle\!\langle x\z^{k-i}\rangle\!\rangle_N\langle\!\langle y\z^{j-k}\rangle\!\rangle_N
\end{equation}
so that Equation~\eqref{h1} is equivalent to
\begin{equation}
\sum_{k=0}^{N-1}
\langle\!\langle x\z^{k-i}\rangle\!\rangle_N\langle\!\langle y\z^{j-k}\rangle\!\rangle_N
=\langle\!\langle xy\z^{j-i}\rangle\!\rangle_N\Leftrightarrow
\sum_{k=0}^{N-1}
\langle\!\langle x\z^{k}\rangle\!\rangle_N\langle\!\langle y\z^{j-k}\rangle\!\rangle_N
=\langle\!\langle xy\z^{j}\rangle\!\rangle_N \,.
\end{equation}
The definition~\eqref{double-angle-brackets} gives 
$\langle\!\langle x\rangle\!\rangle_N=N^{-1}\sum_{a=0}^{N-1}x^a$.
Using this equality, we obtain
\begin{equation}
\begin{aligned}
\sum_{k=0}^{N-1}
\langle\!\langle x\z^{k}\rangle\!\rangle_N\langle\!\langle y\z^{j-k}\rangle\!\rangle_N
& =N^{-2}\sum_{k,a,b=0}^{N-1}(x\z^k)^a(y\z^{j-k})^b
=N^{-1}\sum_{a,b=0}^{N-1}x^ay^b\z^{jb}\delta_{0,\prin{N}{a-b}}
\\ & =N^{-1}\sum_{b=0}^{N-1}(xy\z^j)^b=\langle\!\langle xy\z^{j}\rangle\!\rangle_N \,.
\end{aligned}
\end{equation}

We can diagonalize the matrices $h(x,m)$ by conjugating them by the (discrete)
Fourier transformation operator
$
\langle i| F|j\rangle=\z^{i j}.
$
Indeed, assuming $i,j\in\{0,\ldots,N-1\}$, we have
\begin{equation}
\label{fourier-trans-of-h}
\begin{aligned}
\langle i|Fh(x,m)F^{-1}|j\rangle &=
N^{-1}\sum_{i',j'}\z^{i i'-jj'}\langle i'|h(x,m)|j'\rangle
=N^{-1}\sum_{i',j'}\z^{i i'-jj'+(j'-i')m}\langle\!\langle x\z^{j'-i'}\rangle\!\rangle_N
\\ &
=N^{-1}\sum_{i',j'}\z^{i i'-j(j'+i')+j'm}\langle\!\langle x\z^{j'}\rangle\!\rangle_N
=\delta_{i,j}\sum_{j'}\z^{(m-j)j'}\langle\!\langle x\z^{j'}\rangle\!\rangle_N
\\ &
=N^{-1}\delta_{i,j}\sum_{j'}\z^{(m-j)j'}\sum_{a=0}^{N-1} x^a\z^{aj'}
=\delta_{i,j}\sum_{a=0}^{N-1}\delta_{0,\prin{N}{m-j+a}} x^a=\delta_{i,j}x^{\prin{N}{j-m}} \,.
\end{aligned}
\end{equation}
Thus, conjugating the r-matrix~\eqref{r-matrix-coeff-diff-form} by $F$, we can prove
Equations~\eqref{h2} by explicit computation. Indeed, assuming that
$i,j,k,l\in\{0,\dots,N-1\}$, we have
\begin{multline}
\label{fourier-trans-r-mat1}
\langle i,j\vert (F\otimes F)r(x;m,n)(F^{-1}\otimes F^{-1})|k,l\rangle
=N^{-2}\sum_{i',j',k',l'}\z^{ii'+jj'-kk'-ll'} \langle i',j'|r(x;m,n)|k',l'\rangle
\\
=\frac{ \langle\!\langle x\rangle\!\rangle_N}{N^{2}}
\sum_{i',j',k',l'}\z^{ii'+jj'-kk'-ll'+(i'-k'+n)(l'-j')}
\frac{w(x/\z|j'-i'-m)w(x|l'-k'+n)}{w(x/\z |j'-k'+n-m)w(x/\z |l'-i')}
\\
=\frac{ \langle\!\langle x\rangle\!\rangle_N}{N^{2}}
\sum_{i',j',k',l'}\z^{i(i'+k')+j(j'+k')-kk'-l(l'+k')+(i'+n)(l'-j')}
\frac{w(x/\z|j'-i'-m)w(x|l'+n)}{w(x/\z |j'+n-m)w(x/\z |l'-i')}
\\
=\frac{ \langle\!\langle x\rangle\!\rangle_N\delta_{0,\prin{N}{i+j-k-l}}}{N}
\sum_{i',j',l'}\z^{ii'+jj'-ll'+(i'+n)(l'-j')}
\frac{w(x/\z|j'-i'-m)w(x|l'+n)}{w(x/\z |j'+n-m)w(x/\z |l'-i')}
\end{multline}
where, in third equality, we have shifted the summation variables $i',j',l'$ by
$k'$ and, in the last equality, performed the $k'$-summation by using the
formula
\begin{equation}
\sum_{a=0}^{N-1}\z^{ab}=N\delta_{0,\prin{N}{b}}.
\end{equation}
We continue the calculation in~\eqref{fourier-trans-r-mat1} by shifting the
summation variables $l'\mapsto l'+i'$ and $j'\mapsto j'+m-n$ followed by the
shift $i'\mapsto i'-n$:
\begin{equation}
\label{fourier-trans-r-mat2}  
\begin{aligned}
\eqref{fourier-trans-r-mat1} 
&=\frac{ \langle\!\langle x\rangle\!\rangle_N\delta_{0,\prin{N}{i+j-k-l}}}{N}
\sum_{i',j',l'}\z^{i(i'-n)+j(j'+m-n)-l(l'+i'-n)+i'(l'+i'-j'-m)}
\frac{w(x/\z|j'-i')w(x|l'+i')}{w(x/\z |j')w(x/\z |l')}
\\ &
=\frac{ \langle\!\langle x\rangle\!\rangle_N\delta_{0,\prin{N}{i+j-k-l}}}{N\z^{kn-jm}}
\sum_{i',j',l'}\z^{(i -l-m+i')i' + (j-i' ) j' + (i' - l) l' }
\frac{w(x/\z|j'-i')w(x|l'+i')}{w(x/\z |j')w(x/\z |l')}
\\ &
=\frac{ \langle\!\langle x\rangle\!\rangle_N\delta_{0,\prin{N}{i+j-k-l}}}{N\z^{kn-jm}}
\sum_{i'=0}^{N-1}\frac{w(x/\z|-i')w(x|i')}{\z^{(l+m-i-i')i' }}
\sum_{j',l'}\z^{ (j-i' ) j' + (i' - l) l' }
\frac{w(x\z^{-i'-1}|j')w(x\z^{i'}|l')}{w(x/\z |j')w(x/\z |l')}
\\ &
=\frac{ \langle\!\langle x\rangle\!\rangle_N\delta_{0,\prin{N}{i+j-k-l}}}{N\z^{kn-jm}}
\sum_{i'=0}^{N-1}\frac{w(x/\z|-i')w(x|i')}{\z^{(l+m-i-i')i' }}
f(x\z^{-i'-1},x/\z|\z^{j-i'})f(x\z^{i'},x/\z|\z^{i'-l})
\end{aligned}
\end{equation} 
where, in  third equality, we used the addition formula~\eqref{addition-formula-w},
and in the last equality, we use function $f(x,y|z)$ defined by 
\begin{equation}
f(x,y|z):=\sum_{a}\frac{w(x|a)}{w(y|a)}z^a,\qquad  (1-y^N)z^N=1-x^N,
\end{equation}
whose automorphic properties are described in~\cite{KMS}. In particular, we have
the equality
\begin{equation}
f(x\z^a,x/\z|\z^{-b})
=\frac{x^{\prin{N}{b}}}{\langle\!\langle x\rangle\!\rangle_Nw(x|\prin{N}{a})}
\frac{(\z)_{\prin{N}{a}+\prin{N}{b}}}{(\z)_{\prin{N}{a}} (\z)_{\prin{N}{b}}}\qquad (a,b\in\BZ)
\end{equation}
which we can use to proceed in~\eqref{fourier-trans-r-mat2} as follows:
\begin{equation}
\label{fourier-trans-r-ma3}
\begin{aligned}
\eqref{fourier-trans-r-mat2}
&=\frac{\langle\!\langle x\rangle\!\rangle_N\delta_{0,\prin{N}{i+j-k-l}}}{
  \z^{kn-jm}}\sum_{i'=j}^{N-1}\frac{w(x|i')x^{i'-j}}{\z^{(l+m-i-i')i' }}
\frac{(\z)_{N-1-j}}{(\z)_{N-1-i'}(\z)_{i'-j}} f(x\z^{i'},x/\z|\z^{i'-l})
\\ &
=\frac{\langle\!\langle x\rangle\!\rangle_N\delta_{0,\prin{N}{i+j-k-l}}}{\z^{kn-jm}}
\sum_{i'=j}^{N-1}\frac{w(x|i')x^{i'-j}}{\z^{(l+m-i-i')i' }}
\frac{(\bar\z)_{i'}}{(\bar\z)_{j}(\z)_{i'-j}} f(x\z^{i'},x/\z|\z^{i'-l})
\end{aligned}
\end{equation}
where, in first equality, we used the addition formula~\eqref{addition-formula-w}
for simplification, and in the last equality, the property 
\begin{equation}
(\z)_k(\bar\z)_{N-1-k}=N\qquad (k\in\BZ,\,\,0\le k\le N-1).
\end{equation}
Continuing with the second $f$-function in~\eqref{fourier-trans-r-ma3},
we have
\begin{equation}
\label{fourier-trans-r-ma4}
\begin{aligned}
&=\frac{\delta_{0,\prin{N}{i+j-k-l}}}{\z^{kn-jm}}
\sum_{i'=j}^{l}\frac{x^{l-j}}{\z^{(l+m-i-i')i' }}
\frac{(\bar\z)_{i'}}{(\bar\z)_{j}(\z)_{i'-j}}\frac{(\z)_{l}}{(\z)_{i'} (\z)_{l-i'}}
\\ &
=\frac{\delta_{0,\prin{N}{i+j-k-l}}(\z)_{l}x^{l}}{\z^{(n-j)k}(\z)_{j}x^{j}}
\sum_{i'=0}^{l-j}\frac{(-1)^{i'}\z^{(i'-1)i'/2}}{\z^{(l+m-i-j)i' }(\z)_{i'}(\z)_{l-j-i'}}
\\ &
=\frac{\delta_{0,\prin{N}{i+j-k-l}}(\z)_{l}x^{l}(\z^{k-m};\z)_{l-j}}{
  \z^{(n-j)k}(\z)_{j}x^{j}(\z)_{l-j}}
=\frac{\delta_{0,\prin{N}{i+j-k-l}}(\z)_{l}x^{l}(\z)_{\prin{N}{k-m-1}+l-j}}{
  \z^{(n-j)k}(\z)_{j}x^{j}(\z)_{l-j}(\z)_{\prin{N}{k-m-1}}}
\\ &
=\delta_{\prin{N}{i-m-1}+j,\prin{N}{k-m-1}+l}\,\z^{(j-n)k}x^{l-j}
\frac{(\z)_{l}(\z)_{\prin{N}{i-m-1}}}{(\z)_{j}(\z)_{l-j}(\z)_{\prin{N}{k-m-1}}}
\end{aligned}
\end{equation}
where, in third equality, we used the well-known $q$-binomial formula 
\begin{equation}
  (z;q)_s=\sum_{t=0}^s (-z)^t q^{(t-1)t/2}\frac{(q)_s}{(q)_t(q)_{s-t}},
  \qquad (s\in\BN),
\end{equation}
with $q=\z$, $s=l-j$ and $z=\z^{k-m}$; in fourth equality, the identity
\begin{equation}
(\z^a;\z)_b=\frac{(\z)_{\prin{N}{a-1}+b}}{(\z)_{\prin{N}{a-1}}}\qquad
(a \in \BZ, \, b \in \BN),
\end{equation}
and,  in the last equality, the equivalence
\begin{equation}
  \begin{cases} i+j\equiv k+l\pmod N\ \\ 0\le \prin{N}{k-m-1}+l-j\le N-1
    \end{cases} 
\!\! \Leftrightarrow \,\, \prin{N}{i-m-1}+j=\prin{N}{k-m-1}+l \,.
\end{equation}
Equation~\eqref{fourier-trans-r-ma4}, combined with~\eqref{fourier-trans-of-h},
straightforwardly implies relations~\eqref{h2}.

We remark that the R-matrix coefficients in \eqref{fourier-trans-r-ma4},
when specialized to $m=n=-1$,
coincide with the standard colored Jones R-matrix coefficients.  This means that
in this case, the associated knot invariant is the $N$-colored Jones polynomial
specialized to a primitive $N$-th root of unity.   

\subsection{Knot invariants}
\label{sub.knotinv}

We can use the r-matrix given in Equation~\eqref{r1} to construct knot invariants
from their planar projections. This is a standard construction explained in several
places that include~\cite{Jones:statistical, RT:ribbon, Tu:YB, Tu:book}. We will
follow the presentation of~\cite{Ka:longknots} which does not require neither the
theory of Hopf algebras nor the existence of ribbon elements, and uses as a
combinatorial input a generic planar projection of a long knot with no local
extrema oriented from left to right and equal numbers of positive and negative
crossings. Locally, such a planar projection has eight types of crossings
(four positive and four negative ones) 
\begin{equation}
\label{8x}
\begin{tikzpicture}[scale=.5,baseline]
\draw[thick,<-] (0,1) to [out=-90,in=90] (1,0);
\draw[line width=3pt,white] (1,1) to [out=-90,in=90] (0,0);
\draw[thick,<-] (1,1) to [out=-90,in=90] (0,0);
\end{tikzpicture}
\qquad
\begin{tikzpicture}[scale=.5,baseline]
\draw[thick,<-] (1,1) to [out=-90,in=90] (0,0);
\draw[line width=3pt,white] (0,1) to [out=-90,in=90] (1,0);
\draw[thick,<-] (0,1) to [out=-90,in=90] (1,0);
\end{tikzpicture}
\qquad
\begin{tikzpicture}[scale=.5,baseline]
\draw[thick,->] (1,1) to [out=-90,in=90] (0,0);
\draw[line width=3pt,white] (0,1) to [out=-90,in=90] (1,0);
\draw[thick,<-] (0,1) to [out=-90,in=90] (1,0);
\end{tikzpicture}
\qquad
\begin{tikzpicture}[scale=.5,baseline]
\draw[thick,<-] (0,1) to [out=-90,in=90] (1,0);
\draw[line width=3pt,white] (1,1) to [out=-90,in=90] (0,0);
\draw[thick,->] (1,1) to [out=-90,in=90] (0,0);
\end{tikzpicture}
\qquad
\begin{tikzpicture}[scale=.5,baseline]
\draw[thick,->] (0,1) to [out=-90,in=90] (1,0);
\draw[line width=3pt,white] (1,1) to [out=-90,in=90] (0,0);
\draw[thick,->] (1,1) to [out=-90,in=90] (0,0);
\end{tikzpicture}
\qquad
\begin{tikzpicture}[scale=.5,baseline]
\draw[thick,->] (1,1) to [out=-90,in=90] (0,0);
\draw[line width=3pt,white] (0,1) to [out=-90,in=90] (1,0);
\draw[thick,->] (0,1) to [out=-90,in=90] (1,0);
\end{tikzpicture}
\qquad
\begin{tikzpicture}[scale=.5,baseline]
\draw[thick,<-] (1,1) to [out=-90,in=90] (0,0);
\draw[line width=3pt,white] (0,1) to [out=-90,in=90] (1,0);
\draw[thick,->] (0,1) to [out=-90,in=90] (1,0);
\end{tikzpicture}
\qquad
\begin{tikzpicture}[scale=.5,baseline]
\draw[thick,->] (0,1) to [out=-90,in=90] (1,0);
\draw[line width=3pt,white] (1,1) to [out=-90,in=90] (0,0);
\draw[thick,<-] (1,1) to [out=-90,in=90] (0,0);
\end{tikzpicture}
\end{equation}
to which one assigns the suitably ``rotated'' r-matrices shown in Equations (16)-(19)
of~\cite{Ka:longknots}, two types of vertical segments, and two types of local
extrema
\begin{equation}
\label{2cups}
\begin{tikzpicture}[yscale=.5,baseline]
\draw[thick,->] (0,0) to [out=90,in=-90] (0,1);
\end{tikzpicture}
\qquad
\begin{tikzpicture}[yscale=.5,baseline]
\draw[thick,<-] (0,0) to [out=90,in=-90] (0,1);
\end{tikzpicture}
\qquad
\begin{tikzpicture}[xscale=.5,baseline]
\draw[thick,<-] (0,0) to [out=90,in=90] (1,0);
\end{tikzpicture}
\qquad
\begin{tikzpicture}[xscale=.5,baseline=15]
\draw[thick,<-] (0,1) to [out=-90,in=-90] (1,1);
\end{tikzpicture}
\end{equation}
to which ones assigns respectively the identity operators and $\varepsilon$ and
$\eta$ maps of Equation (7) of~\cite{Ka:longknots}. In our case, we fix a primitive
root of unity $\z$ of order $N$, a (discrete spectral parameter)
$n \in \BZ/N\BZ$ and a planar projection $D$ of a long knot $K$. We 
color each arc of $D$ by an element of $\BZ/N\BZ$, place r-matrices at the crossings
(using the fixed spectral parameter at all crossings), according to the rules
described below, and sum over all indices. The result of this contraction is a
complex number that depends on the colors $i$ and $j$ of the outgoing and incoming
arcs of the diagram of the long knot. These complex numbers arrange in a matrix 
\begin{equation}
  \label{DNn}
  \langle D \rangle_{N,n} \in \operatorname{Mat}_{N}(\BZ[N^{-1},\z])
\end{equation}

The r-matrices given below are rigid, i.e., satisfy the Yang-Baxter equation (9)
and the inverse matrix equation (10) of~\cite{Ka:longknots}. The next theorem
follows from the results of~\cite{Ka:longknots}. 

\begin{theorem}
  \label{thm.KNn}
  The state-sum $\langle D \rangle_{N,n}$ depends on the long knot
  $K$ and not on the planar projection $D$ used. 
\end{theorem}

It remains to explain the rules for the r-matrices. 

Starting  with the case with $m=n=0$, we get the following rules for the four types
of positive crossings  
\begin{equation}
\begin{tikzpicture}[scale=1,baseline=10]
\draw[thick,<-] (0,1) to [out=-90,in=90] (1,0);
\draw[line width=3pt,white] (1,1) to [out=-90,in=90] (0,0);
\draw[thick,<-] (1,1) to [out=-90,in=90] (0,0);
\node (sw) at (0,-.2){\tiny $k$};\node (se) at (1,-.2){\tiny $l$};
\node (nw) at (0,1.2){\tiny $j$};\node (ne) at (1,1.2){\tiny $i$};
\end{tikzpicture}
=
\begin{tikzpicture}[scale=1,baseline=10]
\draw[thick,->] (1,1) to [out=-90,in=90] (0,0);
\draw[line width=3pt,white] (0,1) to [out=-90,in=90] (1,0);
\draw[thick,<-] (0,1) to [out=-90,in=90] (1,0);
\node (sw) at (0,-.2){\tiny $j$};\node (se) at (1,-.2){\tiny $k$};
\node (nw) at (0,1.2){\tiny $i$};\node (ne) at (1,1.2){\tiny $l$};
\end{tikzpicture}
=
\begin{tikzpicture}[scale=1,baseline=10]
\draw[thick,->] (0,1) to [out=-90,in=90] (1,0);
\draw[line width=3pt,white] (1,1) to [out=-90,in=90] (0,0);
\draw[thick,->] (1,1) to [out=-90,in=90] (0,0);
\node (sw) at (0,-.2){\tiny $i$};\node (se) at (1,-.2){\tiny $j$};
\node (nw) at (0,1.2){\tiny $l$};\node (ne) at (1,1.2){\tiny $k$};
\end{tikzpicture}
=V_{i,j,k,l}(\z)\z^{k-l},\quad
\begin{tikzpicture}[scale=1,baseline=10]
\draw[thick,<-] (1,1) to [out=-90,in=90] (0,0);
\draw[line width=3pt,white] (0,1) to [out=-90,in=90] (1,0);
\draw[thick,->] (0,1) to [out=-90,in=90] (1,0);
\node (sw) at (0,-.2){\tiny $k$};\node (se) at (1,-.2){\tiny $l$};
\node (nw) at (0,1.2){\tiny $j$};\node (ne) at (1,1.2){\tiny $i$};
\end{tikzpicture}
=V_{k,l,i,j-1}(\bar\z)\z^{j-1-k}
\end{equation}
and for the four types of negative crossings
\begin{equation}
\begin{tikzpicture}[scale=1,baseline=10]
\draw[thick,<-] (1,1) to [out=-90,in=90] (0,0);
\draw[line width=3pt,white] (0,1) to [out=-90,in=90] (1,0);
\draw[thick,<-] (0,1) to [out=-90,in=90] (1,0);
\node (sw) at (0,-.2){\tiny $k$};\node (se) at (1,-.2){\tiny $l$};
\node (nw) at (0,1.2){\tiny $j$};\node (ne) at (1,1.2){\tiny $i$};
\end{tikzpicture}
=
\begin{tikzpicture}[scale=1,baseline=10]
\draw[thick,<-] (0,1) to [out=-90,in=90] (1,0);
\draw[line width=3pt,white] (1,1) to [out=-90,in=90] (0,0);
\draw[thick,->] (1,1) to [out=-90,in=90] (0,0);
\node (sw) at (0,-.2){\tiny $j$};\node (se) at (1,-.2){\tiny $k$};
\node (nw) at (0,1.2){\tiny $i$};\node (ne) at (1,1.2){\tiny $l$};
\end{tikzpicture}
=
\begin{tikzpicture}[scale=1,baseline=10]
\draw[thick,->] (1,1) to [out=-90,in=90] (0,0);
\draw[line width=3pt,white] (0,1) to [out=-90,in=90] (1,0);
\draw[thick,->] (0,1) to [out=-90,in=90] (1,0);
\node (sw) at (0,-.2){\tiny $i$};\node (se) at (1,-.2){\tiny $j$};
\node (nw) at (0,1.2){\tiny $l$};\node (ne) at (1,1.2){\tiny $k$};
\end{tikzpicture}
=V_{i,j,k,l}(\bar\z)\z^{l-k}, \quad
\begin{tikzpicture}[scale=1,baseline=10]
\draw[thick,->] (0,1) to [out=-90,in=90] (1,0);
\draw[line width=3pt,white] (1,1) to [out=-90,in=90] (0,0);
\draw[thick,<-] (1,1) to [out=-90,in=90] (0,0);
\node (sw) at (0,-.2){\tiny $k$};\node (se) at (1,-.2){\tiny $l$};
\node (nw) at (0,1.2){\tiny $j$};\node (ne) at (1,1.2){\tiny $i$};
\end{tikzpicture}
=V_{k,l,i,j-1}(\z)\z^{k-j+1}.
\end{equation}
Note that the weights of the negative crossings are the complex conjugates
of the corresponding weights of the positive crossings.

Switching in now the general colors $m$ and $n$, we get the rules for the
positive crossings
\begin{equation}
\label{4xpositive}
\begin{aligned}
\begin{tikzpicture}[scale=1,baseline=10]
\draw[thick,<-] (0,1) to [out=-90,in=90] (1,0);
\draw[line width=3pt,white] (1,1) to [out=-90,in=90] (0,0);
\draw[thick,<-] (1,1) to [out=-90,in=90] (0,0);
\node (sw) at (0,-.2){\tiny $k$};\node (se) at (1,-.2){\tiny $l$};
\node (nw) at (0,1.2){\tiny $j$};\node (ne) at (1,1.2){\tiny $i$};
\node (upper) at (.3,.2){\tiny $m$};
\node (lowerer) at (.75,.2){\tiny $n$};
\end{tikzpicture}
=
\begin{tikzpicture}[scale=1,baseline=10]
\draw[thick,->] (1,1) to [out=-90,in=90] (0,0);
\draw[line width=3pt,white] (0,1) to [out=-90,in=90] (1,0);
\draw[thick,<-] (0,1) to [out=-90,in=90] (1,0);
\node (sw) at (0,-.2){\tiny $j$};\node (se) at (1,-.2){\tiny $k$};
\node (nw) at (0,1.2){\tiny $i$};\node (ne) at (1,1.2){\tiny $l$};
\node (upper) at (.25,.2){\tiny $n$};
\node (lowerer) at (.75,.2){\tiny $m$};
\end{tikzpicture}
= &
\begin{tikzpicture}[scale=1,baseline=10]
\draw[thick,->] (0,1) to [out=-90,in=90] (1,0);
\draw[line width=3pt,white] (1,1) to [out=-90,in=90] (0,0);
\draw[thick,->] (1,1) to [out=-90,in=90] (0,0);
\node (sw) at (0,-.2){\tiny $i$};\node (se) at (1,-.2){\tiny $j$};
\node (nw) at (0,1.2){\tiny $l$};\node (ne) at (1,1.2){\tiny $k$};
\node (upper) at (.3,.2){\tiny $m$};
\node (lowerer) at (.75,.2){\tiny $n$};
\end{tikzpicture}
=V_{i,j-m,k-n,l}(\z)\z^{k-l-n+(k-i-n)m},
\\ &
\begin{tikzpicture}[scale=1,baseline=10]
\draw[thick,<-] (1,1) to [out=-90,in=90] (0,0);
\draw[line width=3pt,white] (0,1) to [out=-90,in=90] (1,0);
\draw[thick,->] (0,1) to [out=-90,in=90] (1,0);
\node (sw) at (0,-.2){\tiny $k$};\node (se) at (1,-.2){\tiny $l$};
\node (nw) at (0,1.2){\tiny $j$};\node (ne) at (1,1.2){\tiny $i$};
\node (upper) at (.25,.2){\tiny $n$};
\node (lowerer) at (.75,.2){\tiny $m$};
\end{tikzpicture}
=V_{k,l,i-m,j-n-1}(\bar\z)\z^{j-1-k-n+(j-l-n)m}
\end{aligned}
\end{equation}
and the negative crossings
\begin{equation}
\label{4xnegative}
\begin{aligned}
\begin{tikzpicture}[scale=1,baseline=10]
\draw[thick,<-] (1,1) to [out=-90,in=90] (0,0);
\draw[line width=3pt,white] (0,1) to [out=-90,in=90] (1,0);
\draw[thick,<-] (0,1) to [out=-90,in=90] (1,0);
\node (sw) at (0,-.2){\tiny $k$};\node (se) at (1,-.2){\tiny $l$};
\node (nw) at (0,1.2){\tiny $j$};\node (ne) at (1,1.2){\tiny $i$};
\node (upper) at (.25,.2){\tiny $n$};
\node (lowerer) at (.75,.2){\tiny $m$};
\end{tikzpicture}
=
\begin{tikzpicture}[scale=1,baseline=10]
\draw[thick,<-] (0,1) to [out=-90,in=90] (1,0);
\draw[line width=3pt,white] (1,1) to [out=-90,in=90] (0,0);
\draw[thick,->] (1,1) to [out=-90,in=90] (0,0);
\node (sw) at (0,-.2){\tiny $j$};\node (se) at (1,-.2){\tiny $k$};
\node (nw) at (0,1.2){\tiny $i$};\node (ne) at (1,1.2){\tiny $l$};
\node (upper) at (.3,.2){\tiny $m$};
\node (lowerer) at (.75,.2){\tiny $n$};
\end{tikzpicture}
= &
\begin{tikzpicture}[scale=1,baseline=10]
\draw[thick,->] (1,1) to [out=-90,in=90] (0,0);
\draw[line width=3pt,white] (0,1) to [out=-90,in=90] (1,0);
\draw[thick,->] (0,1) to [out=-90,in=90] (1,0);
\node (sw) at (0,-.2){\tiny $i$};\node (se) at (1,-.2){\tiny $j$};
\node (nw) at (0,1.2){\tiny $l$};\node (ne) at (1,1.2){\tiny $k$};
\node (upper) at (.25,.2){\tiny $n$};
\node (lowerer) at (.75,.2){\tiny $m$};
\end{tikzpicture}
=V_{i,j-n,k-m,l}(\bar\z)\z^{l-k+(l-j+1+n)m},
\\ &
\begin{tikzpicture}[scale=1,baseline=10]
\draw[thick,->] (0,1) to [out=-90,in=90] (1,0);
\draw[line width=3pt,white] (1,1) to [out=-90,in=90] (0,0);
\draw[thick,<-] (1,1) to [out=-90,in=90] (0,0);
\node (sw) at (0,-.2){\tiny $k$};\node (se) at (1,-.2){\tiny $l$};
\node (nw) at (0,1.2){\tiny $j$};\node (ne) at (1,1.2){\tiny $i$};
\node (upper) at (.3,.2){\tiny $m$};
\node (lowerer) at (.75,.2){\tiny $n$};
\end{tikzpicture}
=V_{k,l,i-n,j-m-1}(\z)\z^{k-j+1+(k-i+1+n)m} \,.
\end{aligned}
\end{equation}
Note that for general $m,n$, the weights of the negative crossings are not
the complex conjugates of the corresponding weights of the positive crossings.
Note also that the eight r-matrices needed in Equations (16)-(19)
of~\cite{Ka:longknots}, up to powers of $\z$,  are all expressed in terms of
the symbols $V_{i,j,k,l}(\z)$ and their complex conjugates $V_{i,j,k,l}(\bar\z)$.

Additionally, we have relations between colored and uncolored weights
\begin{equation}
\label{4xsym}
\begin{tikzpicture}[scale=1,baseline=10]
\draw[thick,<-] (0,1) to [out=-90,in=90] (1,0);
\draw[line width=3pt,white] (1,1) to [out=-90,in=90] (0,0);
\draw[thick,<-] (1,1) to [out=-90,in=90] (0,0);
\node (sw) at (0,-.2){\tiny $k$};\node (se) at (1,-.2){\tiny $l$};
\node (nw) at (0,1.2){\tiny $j$};\node (ne) at (1,1.2){\tiny $i$};
\node (upper) at (.3,.2){\tiny $m$};
\node (lowerer) at (.75,.2){\tiny $n$};
\end{tikzpicture}
=
\begin{tikzpicture}[scale=1,baseline=10]
\draw[thick,->] (1,1) to [out=-90,in=90] (0,0);
\draw[line width=3pt,white] (0,1) to [out=-90,in=90] (1,0);
\draw[thick,<-] (0,1) to [out=-90,in=90] (1,0);
\node (sw) at (0,-.2){\tiny $j$};\node (se) at (1,-.2){\tiny $k$};
\node (nw) at (0,1.2){\tiny $i$};\node (ne) at (1,1.2){\tiny $l$};
\node (upper) at (.25,.2){\tiny $n$};
\node (lowerer) at (.75,.2){\tiny $m$};
\end{tikzpicture}
= 
\begin{tikzpicture}[scale=1,baseline=10]
\draw[thick,->] (0,1) to [out=-90,in=90] (1,0);
\draw[line width=3pt,white] (1,1) to [out=-90,in=90] (0,0);
\draw[thick,->] (1,1) to [out=-90,in=90] (0,0);
\node (sw) at (0,-.2){\tiny $i$};\node (se) at (1,-.2){\tiny $j$};
\node (nw) at (0,1.2){\tiny $l$};\node (ne) at (1,1.2){\tiny $k$};
\node (upper) at (.3,.2){\tiny $m$};
\node (lowerer) at (.75,.2){\tiny $n$};
\end{tikzpicture}
=
\begin{tikzpicture}[scale=1,baseline=10]
\draw[thick,<-] (0,1) to [out=-90,in=90] (1,0);
\draw[line width=3pt,white] (1,1) to [out=-90,in=90] (0,0);
\draw[thick,<-] (1,1) to [out=-90,in=90] (0,0);
\node (sw) at (0,-.2){\tiny $k-n$};\node (se) at (1,-.2){\tiny $l$};
\node (nw) at (0,1.2){\tiny $j-m$};\node (ne) at (1,1.2){\tiny $i$};
\end{tikzpicture}
\z^{(k-i-n)m}, \quad
\begin{tikzpicture}[scale=1,baseline=10]
\draw[thick,<-] (1,1) to [out=-90,in=90] (0,0);
\draw[line width=3pt,white] (0,1) to [out=-90,in=90] (1,0);
\draw[thick,->] (0,1) to [out=-90,in=90] (1,0);
\node (sw) at (0,-.2){\tiny $k$};\node (se) at (1,-.2){\tiny $l$};
\node (nw) at (0,1.2){\tiny $j$};\node (ne) at (1,1.2){\tiny $i$};
\node (upper) at (.25,.2){\tiny $n$};
\node (lowerer) at (.75,.2){\tiny $m$};
\end{tikzpicture}
=
\begin{tikzpicture}[scale=1,baseline=10]
\draw[thick,<-] (1,1) to [out=-90,in=90] (0,0);
\draw[line width=3pt,white] (0,1) to [out=-90,in=90] (1,0);
\draw[thick,->] (0,1) to [out=-90,in=90] (1,0);
\node (sw) at (0,-.2){\tiny $k$};\node (se) at (1,-.2){\tiny $l$};
\node (nw) at (0,1.2){\tiny $j-n$};\node (ne) at (1,1.2){\tiny $i-m$};
\end{tikzpicture}
\z^{(j-l-n)m}
\end{equation}
in the case of positive crossings, and
\begin{equation}
\label{4xsym2}
\begin{tikzpicture}[scale=1,baseline=10]
\draw[thick,<-] (1,1) to [out=-90,in=90] (0,0);
\draw[line width=3pt,white] (0,1) to [out=-90,in=90] (1,0);
\draw[thick,<-] (0,1) to [out=-90,in=90] (1,0);
\node (sw) at (0,-.2){\tiny $k$};\node (se) at (1,-.2){\tiny $l$};
\node (nw) at (0,1.2){\tiny $j$};\node (ne) at (1,1.2){\tiny $i$};
\node (upper) at (.25,.2){\tiny $n$};
\node (lowerer) at (.75,.2){\tiny $m$};
\end{tikzpicture}
=
\begin{tikzpicture}[scale=1,baseline=10]
\draw[thick,<-] (0,1) to [out=-90,in=90] (1,0);
\draw[line width=3pt,white] (1,1) to [out=-90,in=90] (0,0);
\draw[thick,->] (1,1) to [out=-90,in=90] (0,0);
\node (sw) at (0,-.2){\tiny $j$};\node (se) at (1,-.2){\tiny $k$};
\node (nw) at (0,1.2){\tiny $i$};\node (ne) at (1,1.2){\tiny $l$};
\node (upper) at (.3,.2){\tiny $m$};
\node (lowerer) at (.75,.2){\tiny $n$};
\end{tikzpicture}
= 
\begin{tikzpicture}[scale=1,baseline=10]
\draw[thick,->] (1,1) to [out=-90,in=90] (0,0);
\draw[line width=3pt,white] (0,1) to [out=-90,in=90] (1,0);
\draw[thick,->] (0,1) to [out=-90,in=90] (1,0);
\node (sw) at (0,-.2){\tiny $i$};\node (se) at (1,-.2){\tiny $j$};
\node (nw) at (0,1.2){\tiny $l$};\node (ne) at (1,1.2){\tiny $k$};
\node (upper) at (.25,.2){\tiny $n$};
\node (lowerer) at (.75,.2){\tiny $m$};
\end{tikzpicture}
=
\begin{tikzpicture}[scale=1,baseline=10]
\draw[thick,<-] (1,1) to [out=-90,in=90] (0,0);
\draw[line width=3pt,white] (0,1) to [out=-90,in=90] (1,0);
\draw[thick,<-] (0,1) to [out=-90,in=90] (1,0);
\node (sw) at (0,-.2){\tiny $k-m$};\node (se) at (1,-.2){\tiny $l$};
\node (nw) at (0,1.2){\tiny $j-n$};\node (ne) at (1,1.2){\tiny $i$};
\end{tikzpicture}
\z^{(l-j+n)m}, \quad
\begin{tikzpicture}[scale=1,baseline=10]
\draw[thick,->] (0,1) to [out=-90,in=90] (1,0);
\draw[line width=3pt,white] (1,1) to [out=-90,in=90] (0,0);
\draw[thick,<-] (1,1) to [out=-90,in=90] (0,0);
\node (sw) at (0,-.2){\tiny $k$};\node (se) at (1,-.2){\tiny $l$};
\node (nw) at (0,1.2){\tiny $j$};\node (ne) at (1,1.2){\tiny $i$};
\node (upper) at (.3,.2){\tiny $m$};
\node (lowerer) at (.75,.2){\tiny $n$};
\end{tikzpicture}
=
\begin{tikzpicture}[scale=1,baseline=10]
\draw[thick,->] (0,1) to [out=-90,in=90] (1,0);
\draw[line width=3pt,white] (1,1) to [out=-90,in=90] (0,0);
\draw[thick,<-] (1,1) to [out=-90,in=90] (0,0);
\node (sw) at (0,-.2){\tiny $k$};\node (se) at (1,-.2){\tiny $l$};
\node (nw) at (0,1.2){\tiny $j-m$};\node (ne) at (1,1.2){\tiny $i-n$};
\end{tikzpicture}
\z^{(k-i+n)m} 
\end{equation}
in the case of negative crossings.

We complement the rules by the weights on the four types of segments
\begin{equation}
  \label{segments}
\begin{tikzpicture}[yscale=.5,baseline=4]
\draw[thick,->] (0,0) to [out=90,in=-90] (0,1);
\node (n) at (0,1.3){\tiny $i$};\node (s) at (0,-.3){\tiny $j$};
\end{tikzpicture}
=
\begin{tikzpicture}[yscale=.5,baseline=4]
\draw[thick,<-] (0,0) to [out=90,in=-90] (0,1);
\node (n) at (0,1.3){\tiny $j$};\node (s) at (0,-.3){\tiny $i$};
\end{tikzpicture}
=
\begin{tikzpicture}[xscale=1,baseline=0]
\draw[thick,<-] (0,0) to [out=90,in=90] (1,0);
\node (w) at (0,-.2){\tiny $i$};\node (e) at (1,-.2){\tiny $j$};
\end{tikzpicture}
=
\begin{tikzpicture}[xscale=1,baseline=25]
\draw[thick,<-] (0,1) to [out=-90,in=-90] (1,1);
\node (w) at (0,1.2){\tiny $i$};\node (e) at (1,1.2){\tiny $j$};
\end{tikzpicture}
=\delta_{i,j}.
\end{equation}

With our rules, one can calculate weights of few simplest composed diagrams
in the colored case
\begin{equation}
\begin{tikzpicture}[yscale=1,baseline=18]
\coordinate (a1) at (0,1);
\coordinate (a2) at (.75,.5);
\coordinate (a3) at (0.25,.5);
\coordinate (a4) at (1,1);
\node at (0,1.2) {\tiny$j$};
\node at (1,1.2) {\tiny$i$};
\node at (.95,0.5) {\tiny$n$};
\draw[thick] (a1) to [out=-45,in=90] (a2);
\draw[thick] (a2) to [out=-90,in=-90] (a3);
\draw[line width=3pt,white] (a3) to [out=90,in=-135] (a4);
\draw[thick,->] (a3) to [out=90,in=-135] (a4);
\end{tikzpicture}
=\delta_{j,\prin{N}{i+1}},\quad 
\begin{tikzpicture}[yscale=1,baseline=18]
\coordinate (a1) at (0,1);
\coordinate (a2) at (.75,.5);
\coordinate (a3) at (0.25,.5);
\coordinate (a4) at (1,1);
\node at (0,1.2) {\tiny$j$};
\node at (1,1.2) {\tiny$i$};
\node at (.95,0.5) {\tiny$n$};
\draw[thick] (a2) to [out=-90,in=-90] (a3);
\draw[thick,->] (a3) to [out=90,in=-135] (a4);
\draw[line width=3pt,white] (a1) to [out=-45,in=90] (a2);
\draw[thick] (a1) to [out=-45,in=90] (a2);
\end{tikzpicture}
=\delta_{j,\prin{N}{i+1}}\z^n
\end{equation}
and
\begin{equation}
\begin{tikzpicture}[yscale=-1,baseline=-25]
\coordinate (a1) at (0,1);
\coordinate (a2) at (.75,.5);
\coordinate (a3) at (0.25,.5);
\coordinate (a4) at (1,1);
\node at (0,1.2) {\tiny$j$};
\node at (1,1.2) {\tiny$i$};
\node at (.95,0.5) {\tiny$n$};
\draw[thick] (a1) to [out=-45,in=90] (a2);
\draw[thick] (a2) to [out=-90,in=-90] (a3);
\draw[line width=3pt,white] (a3) to [out=90,in=-135] (a4);
\draw[thick,->] (a3) to [out=90,in=-135] (a4);
\end{tikzpicture}
=\delta_{i,\prin{N}{j+1}},\quad 
\begin{tikzpicture}[yscale=-1,baseline=-25]
\coordinate (a1) at (0,1);
\coordinate (a2) at (.75,.5);
\coordinate (a3) at (0.25,.5);
\coordinate (a4) at (1,1);
\node at (0,1.2) {\tiny$j$};
\node at (1,1.2) {\tiny$i$};
\node at (.95,0.5) {\tiny$n$};
\draw[thick] (a2) to [out=-90,in=-90] (a3);
\draw[thick,->] (a3) to [out=90,in=-135] (a4);
\draw[line width=3pt,white] (a1) to [out=-45,in=90] (a2);
\draw[thick] (a1) to [out=-45,in=90] (a2);
\end{tikzpicture}
=\delta_{i,\prin{N}{j+1}}\z^{-n}
\end{equation}
which imply that
\begin{equation}
\begin{tikzpicture}[yscale=1,baseline=27]
\coordinate (a1) at (0,1);
\coordinate (a2) at (.75,.5);
\coordinate (a3) at (0.25,.5);
\coordinate (a4) at (1,1);
\node at (1,1.2) {\tiny$i$};
\draw[thick] (a2) to [out=-90,in=-90] (a3);
\draw[thick,->] (a3) to [out=90,in=-135] (a4);
\draw[line width=3pt,white] (a1) to [out=-45,in=90] (a2);
\draw[thick] (a1) to [out=-45,in=90] (a2);
\begin{scope}[yscale=-1,
xshift=-1cm,yshift=-2cm]
\coordinate (a1) at (0,1);
\coordinate (a2) at (.75,.5);
\coordinate (a3) at (0.25,.5);
\coordinate (a4) at (1,1);
\node at (0,1.2) {\tiny$j$};
\node at (.95,0.5) {\tiny$n$};
\draw[thick] (a1) to [out=-45,in=90] (a2);
\draw[thick] (a2) to [out=-90,in=-90] (a3);
\draw[line width=3pt,white] (a3) to [out=90,in=-135] (a4);
\draw[thick,->] (a3) to [out=90,in=-135] (a4);
\end{scope}
\end{tikzpicture}
=\delta_{i,j}\z^{n},\quad
\begin{tikzpicture}[yscale=1,baseline=27]
\coordinate (a1) at (0,1);
\coordinate (a2) at (.75,.5);
\coordinate (a3) at (0.25,.5);
\coordinate (a4) at (1,1);
\node at (1,1.2) {\tiny$i$};
\draw[thick] (a1) to [out=-45,in=90] (a2);
\draw[thick] (a2) to [out=-90,in=-90] (a3);
\draw[line width=3pt,white] (a3) to [out=90,in=-135] (a4);
\draw[thick,->] (a3) to [out=90,in=-135] (a4);

\begin{scope}[yscale=-1,
xshift=-1cm,yshift=-2cm]
\coordinate (a1) at (0,1);
\coordinate (a2) at (.75,.5);
\coordinate (a3) at (0.25,.5);
\coordinate (a4) at (1,1);
\node at (0,1.2) {\tiny$j$};
\node at (.95,0.5) {\tiny$n$};
\draw[thick] (a2) to [out=-90,in=-90] (a3);
\draw[thick,->] (a3) to [out=90,in=-135] (a4);
\draw[line width=3pt,white] (a1) to [out=-45,in=90] (a2);
\draw[thick] (a1) to [out=-45,in=90] (a2);
\end{scope}
\end{tikzpicture}
=\delta_{i,j}\z^{-n} \,.
\end{equation}

It is interesting to note, that the continuous spectral parameter can be used in
an alternative formulation of knot and link invariants along the idea of
Jones~\cite{Jones:statistical} of considering a ``statistical mechanics'' model
on piecewise linear diagrams with straight segments where the multiplicative
angles are encoded in the continuous spectral parameters. In our case, such an
approach might be applicable only in the case with color $n\equiv-1\pmod N$ but
not for other values of $n$.  In the case $n\equiv-1\pmod N$, the starting
graphical rules would look like 

\begin{equation}
\begin{tikzpicture}[scale=1,baseline=10]
\draw[thick,<-] (0,1) -- (1,0);
\draw[line width=3pt,white] (1,1) -- (0,0);
\draw[thick,<-] (1,1) -- (0,0);
\node (sw) at (0,-.2){\tiny $k$};\node (se) at (1,-.2){\tiny $l$};
\node (nw) at (0,1.2){\tiny $j$};\node (ne) at (1,1.2){\tiny $i$};
\draw (.59-.3 ,.59-.3 ) arc (-135:-45:.3);
\node (angle) at (.5,.1){\tiny $x$};
\end{tikzpicture}
=\langle i,j|r(x;-1,-1)|k,l\rangle,\quad 
\begin{tikzpicture}[scale=1,baseline=10]
\draw[thick,->] (0,0) -- (.5,0.5)--(1,.5);
\node (nw) at (-0.1,0){\tiny $j$};\node (ne) at (1.1,.5){\tiny $i$};
\draw (.7,.5 ) arc (0:-135:.2);
\node (angle) at (.6,.2){\tiny $x$};
\end{tikzpicture}
=\langle i|h(-1/x,0)|j\rangle \,,
\end{equation}
where the continuous variables $x$'s are thought of as exponentiated angles $e^{i\alpha}$,
whose product around each vertex is $1$.

\subsection{An example}
\label{sub.ex41}

In this sub-section we write an explicit expression for the invariant
$\langle 4_1 \rangle_{N,n}$ of the simplest hyperbolic knot $4_1$ in the case of
a primitive $N$-th root of unity $\z$ and a color $n \in \BZ/N\BZ$.
The invariant is an $N \times N$ matrix whose entries
$M_{i,j}:=\langle i|\langle 4_1\rangle_{N,n}|j\rangle$ are given by

\begin{equation}
\label{41inv} 
\begin{aligned}
M_{i,j} &=
\sum_{k_1,\dots,k_7}
\begin{tikzpicture}[scale=.4,baseline=55]
\coordinate (a0) at (0,0);
\coordinate (a1) at (4,7);
\coordinate (a2) at (2,9);
\coordinate (a3) at (0,7);
\coordinate (a4) at (4,4);
\coordinate (a5) at (0,4);
\coordinate (a6) at (2,7);
\coordinate (a7) at (0,10);
\draw[thick] (a1) to [out=90,in=0] (a2);
\draw[thick,->] (a4) to [out=-90,in=-90] (a5);
\draw[thick,->]  (a6) to [out=90,in=-90] (a7);
\draw[line width=3pt,white] (a0) to [out=90,in=-90] (a1);
\draw[thick] (a0) to [out=90,in=-90] (a1);
\draw[line width=3pt,white] (a2) to [out=180,in=90] (a3);
\draw[thick,->] (a2) to [out=180,in=90] (a3);
\draw[line width=3pt,white] (a3) to [out=-90,in=90] (a4);
\draw[thick,->] (a3) to [out=-90,in=90] (a4);
\draw[line width=3pt,white] (a5) to [out=90,in=-90] (a6);
\draw[thick,->]  (a5) to [out=90,in=-90] (a6);
\node at (.3,.4){\tiny $j$};
\node  at (.3,9.6){\tiny $i$};
\node at (3,3.8){\tiny $k_1$};\node  at (3,9.2){\tiny $k_2$};
\node at (.6,6.6){\tiny $k_3$};\node  at (2.4,5.75){\tiny $k_4$};
\node  at (3,2.5){\tiny $k_5$};\node at (.6,3.8){\tiny $k_6$};
\node  at (2.4,7.5){\tiny $k_7$};
\node at (.8,1.5){\tiny $n$};
\end{tikzpicture}
=\sum_{k_1,\dots,k_7}
\begin{tikzpicture}[baseline=12]
\node at (0,2){$ V_{i,k_3-n,k_7-n,k_2}(\bar\z)\z^{k_2-k_7+(k_2-k_3+1+n)n}$};
\node at (0,1){$ V_{k_6,k_4,k_7-n,k_3-n-1}(\z)\z^{k_6-k_3+1+(k_6-k_7+1+n)n}$};
\node at (0,0){$ V_{k_1,k_5,k_2-n,k_4-n-1}(\bar\z)\z^{k_4-k_1-1+(k_4-k_5-1-n)n}$};
\node at (0,-1){$ V_{k_1,k_6-n,j-n,k_5}(\z)\z^{j-k_5+(j-k_1-1-n)n}$};
\end{tikzpicture}
\\ &
=\sum_{k_1,\dots,k_7} V_{i,k_3-n,k_7-n,k_2}(\bar\z)V_{k_6,k_4,k_7-n,k_3-n-1}(\z)
V_{k_1,k_5,k_2-n,k_4-n-1}(\bar\z)V_{k_1,k_6-n,j-n,k_5}(\z)
\\ &
\qquad \times\z^{(j-k_1+k_2-k_3+k_4-k_5+k_6-k_7)(n+1)}
\\ &
= \sum_{k_1,\dots,k_7}V_{i,k_3,k_7,k_2}(\bar\z)V_{k_6,k_4,k_7,k_3-1}(\z)
V_{k_1,k_5,k_2-n,k_4-n-1}(\bar\z)V_{k_1,k_6-n,j-n,k_5}(\z)
\\ &
\qquad \times\z^{(j-k_1+k_2-k_3+k_4-k_5+k_6-k_7-2n)(n+1)}
\\ &
= \sum_{k_1,\dots,k_7}V_{0,k_3,k_7,k_2}(\bar\z)V_{k_6,k_4,k_7,k_3-1}(\z)
V_{k_1,k_5,k_2-n,k_4-n-1}(\bar\z)V_{k_1,k_6-n,j-i-n,k_5}(\z)
\\ &
\qquad \times\z^{(j-i-k_1+k_2-k_3+k_4-k_5+k_6-k_7-2n)(n+1)}
\\ &
= \sum_{k_1,\dots,k_7}V_{0,k_3,k_7,k_2}(\bar\z)V_{0,k_4,k_7-k_6,k_3-k_6-1}(\z)
V_{0,k_5,k_2+k_1-k_6,k_4+k_1-1}(\bar\z)
\\ &
\qquad \times V_{0,k_1,j-i-k_6+k_1,k_5}(\z)
\z^{(j-i+2k_1+k_2-k_3+k_4-k_5-k_7)(n+1)} \,.
\end{aligned}
\end{equation}

The above is a sum over $(\BZ/N\BZ)^7$ and it is not clear how to
simplify to sum in fewer variables. Numerical calculations suggest that
$\langle K \rangle_{N,n}$ is a multiple of the identity matrix. 

\subsection{A conjecture}
\label{sub.conj}

When $n=-1$, Equation~\eqref{fourier-trans-r-ma4} implies that the invariant of a long knot $K$ is a multiple 
of the identity $N\times N$ matrix, and the multiple is the specialization of the $N$-colored Jones polynomial that enters the Volume
Conjecture~\cite{K95,K97}. When $n \in \BZ/N\BZ$ is arbitrary,
we conjecture the following. 

\begin{conjecture}
\label{conj.2}  
For all knots $K$, strictly positive integers $N$ and $n$, and a primitive
$N$-th root of unity $\z$, we have
\be
\label{conj2}
\langle K \rangle_{N,n} = J^K_{n+1}(\z) 1_N 
\ee
where $1_N$ denotes the identity $N \times N$ matrix.
\end{conjecture}
In particular, $\langle K \rangle_{N,0}=1_N$ is a trivial invariant coming from
a nontrivial r-matrix over an $N$-dimensional vector space.


\bibliographystyle{hamsalpha}
\bibliography{biblio}
\end{document}